\documentclass[12pt]{amsart}
\usepackage{amsmath,amsfonts,amssymb,amsthm,bbm,latexsym,mathrsfs}
\usepackage{graphicx,color,epsfig,fancyhdr,dsfont}
\usepackage{hyperref}
\usepackage{todonotes}
\usepackage{a4wide}
\hypersetup{
    colorlinks,
    linkcolor={red!50!black},
    citecolor={blue!50!black},
    urlcolor={blue!80!black}
}

\newcommand{\N}{{\mathbb N}}

\def\U{\mathcal{U}}

\newcommand{\supp}{\mathrm{supp}}

\newtheorem{theorem}{Theorem}[section]
\newtheorem{lemma}[theorem]{Lemma}

\theoremstyle{definition}

\newcommand{\id}{{\ensuremath{\mathds{1}}}}

\numberwithin{equation}{section}

\RequirePackage{amsmath}
\RequirePackage{amssymb}
\RequirePackage{amsthm}
\RequirePackage{epsfig}

\begin{document}

\title[Improved not feeling the boundary estimates]
{Heat Kernel estimates for general boundary problems}

\author{Liangpan Li and Alexander Strohmaier}

\dedicatory{Dedicated to the memory of Yuri Safarov (1958--2015)}

\address{Department of Mathematical Sciences,
Loughborough University, LE11 3TU, UK}
 \email{l.li@lboro.ac.uk, a.strohmaier@lboro.ac.uk}



\begin{abstract}
We show that not feeling the boundary estimates for heat kernels hold for any non-negative self-adjoint 
extension of the Laplace operator acting on vector-valued 
compactly supported functions on a domain in $\mathbb{R}^d$. 
They are therefore valid for any choice of boundary condition and we show that
the implied constants can be chosen independent of the self-adjoint extension.
The method of proof is very general and is based on finite propagation speed 
estimates and explicit Fourier Tauberian theorems obtained by Y. Safarov.
\end{abstract}

\maketitle


\section{Introduction}

Let $U$ be an open  set in $\mathbb{R}^d$ ($d\geq2$) and consider the Dirichlet Laplace operator $\Delta_D$
on $L^2(U)$. Then the fundamental solution $K(t), \;t \geq 0$ of the heat equation with Dirichlet boundary conditions
can be constructed via spectral calculus as
$$
 K_D(t) = \exp(- t\Delta_D).
$$
The integral kernel $K_D(x,y;t)$ of $K_D(t)$ defined by
$$
 (K_D(t) f)(x) = \int_U K_D(x,y;t) f(y) dy
$$
is a positive smooth function on $U \times U \times \mathbb{R}^+$. 
It describes the propagation of heat from the point $x$ to the point $y$ in time $t$.
In case $U=\mathbb{R}^d$ the heat kernel is explicitly given by
\[K_{0}(x,y;t)=(4\pi t)^{-d/2} \exp({-\frac{|x-y|^2}{4t}}).\]
On physical grounds one expects that for small times the heat kernel is dominated by local contributions 
that do not involve the boundary of $U$.
This is essentially the principle of not feeling the boundary by Kac (\cite{Kac}). 
Both 
qualitative and explicit quantitative versions of this principle have been obtained by some authors (\cite{Berg1,Berg2,Hunt}) by exploiting
the probabilistic interpretation of the heat kernel (\cite{Ray54,Rosenblatt,Simon}). The best estimate for 
the Dirichlet Laplacian we are aware of was obtained in
\cite{Berg2} and reads
\begin{equation}\label{BestBerg}
 1 \geq \frac{K_D(x,y;t)}{K_0(x,y;t)} \geq 1- e^{-\delta^2/t} \sum_{j=1}^d \frac{2^j}{(j-1)!} \left( \frac{\delta^2}{t}\right)^{j-1}.
\end{equation}
Here $\delta$ is the distance of the convex hull of $\{ x, y\}$ to the boundary $\partial U$ of $U$. 
It is also known (see e.g. \cite{Berg2,Hsu}) that 
\[
 \lim_{t \to 0_+} t \log \left( 1-  \frac{K_D(x,x;t)}{K_0(x,x;t)} \right) = - \rho^2(x),
\]
where $\rho(x)$ is the distance of $x$ to $\partial U$.
These estimates show that as $t$ goes to $0$ the error in approximating the heat kernel by $K_0(x,y;t)$ is exponentially small with decay rate determined by the distance to the boundary.

Explicit estimates like these are important in spectral geometry and mathematical physics.
For example the meromorphic extension of the local spectral zeta function
is usually based on the expansion of the heat kernel (\cite{Gilkey}). 
The above estimates directly lead to bounds on the local spectral zeta functions or other spectral invariants (\cite{Elizalde}).
A particular example of such a local spectral function is the Casimir energy density that plays a distinguished role in physics.
For these applications it is important to allow for boundary conditions other than Dirichlet. For example Casimir interaction
between two conducting obstacles is described by the Casimir energy density of the photon field. This is obtained from the Laplace operator
acting on one forms with relative boundary conditions. 
To be able to deal with Laplace operators of such type one needs to consider
self-adjoint extensions of the vector-valued Laplace operator on domains that are not simply sums of Laplace operator on functions.
In order to illustrate this let us discuss briefly the example of the propagation of electromagnetic waves, or in a quantum field theoretic description the propagation
of a photon. To keep things simple assume that $U$ is either $\mathbb R^3 \backslash K$, where $K \subset \mathbb R^3$ is a compact subset with smooth 
boundary, or a bounded domain with smooth boundary. If the boundary is modelled as a perfect conductor then separation of variables in the wave equation results in the Laplace operator
acting on $\mathbb C^3$-valued functions and the following boundary conditions for the electromagnetic vector potential $\mathbf{A}(x)$ of the form
$$
 \mathbf{n}(x) \times \mathbf{A}(x) = 0, \quad \nabla  \mathbf{A}(x) = 0
$$
for all $x \in \partial U$. Here $ \mathbf{n}(x)$ is the outward pointing unit-vector-field on the boundary $\partial U$ of $U$.
These boundary conditions define a self-adjoint extension of the Laplace operator acting on $\mathbb C^3$-valued smooth compactly 
supported functions. This self-adjoint extension is however not simply a sum of operators acting on function, as the boundary conditions are different for the different components. 
In fact, the wave group as well as the heat semi-group will in general mix the components of the vector they are acting on.  In physics this corresponds to the fact that mirrors
change the polarization of light.
We would like to refer the reader to \cite{Bordag} for further details and references on Casimir energy
density computations.

The aim of this note is to show that explicit not feeling the boundary estimates can be obtained for any self-adjoint extension of the Laplace operator
acting on vector-valued functions on a domain. 
They can be derived from a combination of finite propagation speed estimates and explicit Fourier Tauberian theorems
that were found by Safarov in \cite{Safarov}. The idea of using finite propagation speed estimates in this context is not new and is already present 
in the classical paper
\cite{Cheeger}. It has since been used by many authors to derive heat kernel bounds on manifolds (see e.g. \cite{Coulhon,Dodziuk,LS,PV,Sikora}). The combination with
the estimates of the spectral function obtained via Fourier Tauberian arguments gives bounds that in some regimes are better than the known estimates for the Dirichlet
heat kernel. The implied constants are independent of the boundary conditions.

To describe the main result let us assume that, as before, $U$ is an open  set in $\mathbb{R}^d, \;d\geq2$ and denote by  $\rho(x)$  the distance from $x\in U$
to the boundary of $U$.
Let $N$ be a positive integer. Consider  in the Hilbert space $L^{2}(U;\mathbb{C}^N)$ 
an arbitrary  non-negative self-adjoint extension $\Delta_U$ of  the  Laplacian
 \[-\big(\frac{\partial^2}{\partial x_1^2}+\cdots+\frac{\partial^2}{\partial x_d^2}\big):C_c^{\infty}(U;\mathbb{C}^N)\rightarrow C_c^{\infty}(U;\mathbb{C}^N)\]
 acting component-wise.
 The heat kernel  for $\Delta_U$, denoted by
 \[\mathbf{K}_U(x,y;t)=\left(\begin{array}{ccc}
K_U^{(11)}(x,y;t) & \cdots & K_U^{(1N)}(x,y;t)\\
\vdots& \ddots &\vdots\\ K_U^{(N1)}(x,y;t) & \cdots & K_U^{(NN)}(x,y;t)
 \end{array}
 \right),\]  
is the integral kernel of  $e^{-t\Delta_U}, \;t>0$ defined by the functional calculus of self-adjoint operators.
When $U=\mathbb{R}^d$  the counterpart for
$\Delta_U$  and $\mathbf{K}_{U}(x,y;t)$ is denoted respectively by $\Delta_0$  and $\mathbf{K}_{0}(x,y;t)$. 
Of course,
\[\mathbf{K}_{0}(x,y;t)=(4\pi t)^{-d/2} \exp({-\frac{|x-y|^2}{4t}}) \id.\]

\begin{theorem}\label{Theorem11}
There exist two positive constants $C_1,C_2$ depending only $d$ such that
if $t\leq\frac{(\rho(x)+\rho(y))^2}{8}$ then
\begin{align*} \|\mathbf{K}_{U}(x,y;t)-\mathbf{K}_{0}(x,y;t)\| 
 \leq
\big(C_1\rho(x,y)^{-d}+C_2\big)\cdot \frac{\displaystyle\exp\left(-\frac{(\rho(x)+\rho(y))^2}{4t}\right)}{\displaystyle t^{2\lceil\frac{d+1}{2}\rceil-\frac{1}{2}}}.
\end{align*}
Here $\rho(x,y)=\min(\rho(x),\rho(y))$.
\end{theorem}

The constants $C_1,C_2$ can be explicitly given, but we refer the reader to the relevant section of this paper for the full description.
As a corollary, we are able to answer a question raised in \cite{Lacey} about an upper estimate for the Neumann heat kernel.

\section{Vector-valued Laplacians}


Throughout  we fix some notations:
 Let $m\in\mathbb{N}$ with $m>\frac{d}{2}$.  Let $V$ denote either $U$ or $0$. 
 Let $\mathbf{G}_V^{(m)}$ denote the (distributional) integral kernel of the operator $(1+\Delta_V)^{-m}$. If $N=1$ we also write  $G_V^{(m)}$ for 
 $\mathbf{G}_V^{(m)}$. By (local) elliptic regularity $\mathbf{G}_V^{(m)}$ is continuous on the open set $U \times U$.
For any $R>0$ define
\begin{equation}\label{formula12}
J_m(R;t)=\inf_{\psi\in A_R}J_m(\psi;t)\ \ \ (R>0),
\end{equation}
where $A_R$ is the set of real-valued functions $\psi$ in $C^{2m}(\mathbb{R})$ such that
$
\mathrm{Supp}(1-\psi)\subset(-R,R)
$, and
\begin{equation}\label{formula11}
J_{m}(\psi;t)=\int_{\mathbb{R}}\Big|(1-\frac{d^2}{ds^2})^{m}\left(\psi(s)e^{-\frac{s^2}{4t}}\right)\Big|ds.
\end{equation} 
Any matrix of size $N\times N$ can be naturally regarded as a linear operator on the Hilbert space $\mathbb{C}^{N}$, so we let
$\|\cdot\|$ denote its operator norm.

\subsection{Finite propagation speed}\label{sub21}

Before we start let us make some notational remarks.
Let  $x,y\in U$, $v,w\in\mathbb{C}^N$.
We denote $\delta_x^{(v)}=\delta_x\otimes v$, where $\delta_x$ is the Dirac delta distribution at $x$.
Strictly speaking, $\delta_x\otimes v$ is not in the domain of the self-adjoint operators $e^{-t\Delta_V}$ and $\cos(s\sqrt{\Delta_V})$.
We understand however expressions such as $\cos(s\sqrt{\Delta_V})(\delta_x\otimes v)$ as distributions (in $x$) with values in the Hilbert space
$L^2(\U; \mathbb{C}^N)$. Pairing with the test function $\varphi \in C^\infty_c(U)$
is defined as $\cos(s\sqrt{\Delta_V})(\varphi \otimes v)$. As usual, the expression $\cos(s\sqrt{\Delta_V})(\delta_x)$
is then understood as a distribution with values in $L^2(\U; \mathbb{C}^N \otimes (\mathbb{C}^N)^*)= L^2(\U; \mathrm{Mat}(N,\mathbb C))$
With this notation the distributional integral kernel $\mathbf{k} \in \mathcal{D}'(U \times U; \mathrm{Mat}(N,\mathbb C))$
of an operator $K$ is $\mathbf{k}(x,y)=\langle \delta_x, K \delta_y \rangle$.
In particular, expressions of the form $\langle \delta_x^{(v)}, K \delta_y^{(w)} \rangle$ are bi-distributions and the pairing with test
functions  $\varphi_1 \otimes \varphi_2 \in C^\infty_c(U \times U)$ is given by
$\langle \varphi_1 \otimes v, K \varphi_2 \otimes w \rangle = \langle v, \mathbf{k} \;w \rangle_{\mathbb{C}^N}(\varphi_1 \otimes \varphi_2)$.

 Alternatively, one can also understand $\cos(s\sqrt{\Delta_V})(\delta_x\otimes v)$ as the distributional limit
of a sequence $\cos(s\sqrt{\Delta_V})(\varphi_n \otimes v)$, where $\varphi_n$ is a $\delta$-family centered at $x$. 
Note that $\cos(s\sqrt{\Delta_V})$ is formally self-adjoint, and a continuous map from $C_c^{\infty}(U;\mathbb{C}^N)$
to $C^{\infty}(U;\mathbb{C}^N)$. This follows from (local) elliptic regularity
because $\Delta^m_V \cos(s\sqrt{\Delta_V})=  \cos(s\sqrt{\Delta_V}) \Delta^m_V$ is a continuous map from $C_0^{\infty}(\U) \to L^2(\U)$ for any $m \in \N$. 
 It therefore extends by duality to a continuous map from
$\mathcal{E}'(U;\mathbb{C}^N)$ to $\mathcal{D}'(U;\mathbb{C}^N)$. As usual, here $\mathcal{D}'(U;\mathbb{C}^N)$ denotes the space of distributions
with values in $\mathbb{C}^N$, and $\mathcal{E}'(U;\mathbb{C}^N)$ denotes the subspace of distributions of compact support.

\begin{theorem}\label{Theorem21} 
The following pointwise estimate holds for the heat kernel:
\begin{align*}
\|\mathbf{K}_U(x,y;t)-\mathbf{K}_0(x,y;t)\| & \\ \leq
\Big(\big(\|\mathbf{G}_U^{(m)}(x,x)\| & \|\mathbf{G}_U^{(m)}(y,y)\|\big)^{\frac{1}{2}}  +\frac{\Gamma(m-\frac{d}{2})}{(4\pi)^{\frac{d}{2}}(m-1)!}\Big)\frac{J_{m}(\rho(x)+\rho(y);t)}{2\sqrt{\pi t}}.
\end{align*}
\end{theorem}

\begin{proof}
Let $\psi\in A_R$ where $R=\rho(x)+\rho(y)$.
 It is well-known (see e.g. \cite{Taylor}) that
\begin{align}
e^{-t\Delta_V}&=\frac{1}{2\sqrt{\pi t}}\int_{\mathbb{R}}\cos(s\sqrt{\Delta_V})e^{-\frac{s^2}{4t}}ds\nonumber\\
&=\frac{1}{2\sqrt{\pi t}}\int_{\mathbb{R}}\cos(s\sqrt{\Delta_V})(1-\psi(s))e^{-\frac{s^2}{4t}}ds\nonumber\\
&\ \ \ \ 
+\frac{1}{2\sqrt{\pi t}}\int_{\mathbb{R}}(1+\Delta_V)^{-m}\cos(s\sqrt{\Delta_V})(1-\frac{d^2}{ds^2})^m(\psi(s)e^{-\frac{s^2}{4t}})ds.\label{221}
\end{align}
Finite propagation speed for the wave equation implies that if $|s_1|<\rho(x)$ then 
$\cos(s_1\sqrt{\Delta_0})\delta_x^{(v)}$ 
has compact support in $U$ and  agrees with $\cos(s_1\sqrt{\Delta_U})\delta_x^{(v)}$. 
Note that any $s\in\mathbb{R}$ with $|s|<\rho(x)+\rho(y)$ can be written as $s=s_1+s_2$ with
$|s_1|<\rho(x)$, $|s_2|<\rho(y)$,  $s_1s_2\geq0$. With this decomposition available and  by considering
\[\cos(s\sqrt{\Delta_V})=2\cos(s_1\sqrt{\Delta_V})\cos(s_2\sqrt{\Delta_V})-\cos((s_1-s_2)\sqrt{\Delta_V})\cos(0\sqrt{\Delta_V})\]
as well as $|s_1-s_2|<\max\{\rho(x),\rho(y)\}$, $0<\min\{\rho(x),\rho(y)\}$, one obtains
\[\langle\delta_x^{(v)},\cos(s\sqrt{\Delta_U})\delta_y^{(w)}\rangle-\langle\delta_x^{(v)},\cos(s\sqrt{\Delta_0})\delta_y^{(w)}\rangle=0\]
for any $s\in\mathbb{R}$ with $|s|<\rho(x)+\rho(y)$. As $\supp(1-\psi)\subset(-\rho(x)-\rho(y),\rho(x)+\rho(y))$, we get
\begin{equation}
(1-\psi(s))\langle\delta_x^{(v)},\cos(s\sqrt{\Delta_U})\delta_y^{(w)}\rangle-(1-\psi(s))\langle\delta_x^{(v)},\cos(s\sqrt{\Delta_0})\delta_y^{(w)}\rangle=0
\label{222}
\end{equation}
for any $s\in\mathbb{R}$. On the other hand, note that
\begin{align*}
\langle\delta_x^{(v)},(1+\Delta_V)^{-m}\cos(s\sqrt{\Delta_V})\delta_y^{(w)}\rangle=\langle(1+\Delta_V)^{-\frac{m}{2}}\delta_x^{(v)},
\cos(s\sqrt{\Delta_V})(1+\Delta_V)^{-\frac{m}{2}}\delta_y^{(w)} \rangle.
\end{align*}
Applying the Cauchy-Schwarz inequality several times this gives
\begin{equation}
|\langle\delta_x^{(v)},(1+\Delta_V)^{-m}\cos(s\sqrt{\Delta_V})\delta_y^{(w)}\rangle|\leq|v||w|\big(\|\mathbf{G}_V^{(m)}(x,x)\|\|\mathbf{G}_V^{(m)}(y,y)\|\big)^{1/2}\label{223}
\end{equation}
for any $s\in\mathbb{R}$. Combining (\ref{222}), (\ref{223}), (\ref{29}) with (\ref{221}) suffices to conclude the proof.
\end{proof}

\subsection{Safarov's estimate}\label{sub22}

Since $\Delta_U$ is a non-negative self-adjoint operator, by the spectral theorem
\[\Delta_U=\int_{0}^{\infty}\lambda\; d\Pi(\lambda),\]
where  $\Pi(\lambda)$ ($\lambda\geq0$) denotes
the spectral projection of $\Delta_U$ onto the interval $[0,\lambda]$.
The so-called spectral function $\mathbf{e}(x,y;\lambda)$, defined to be the integral kernel of
$\Pi(\lambda)$, is smooth in $U\times U$ for each fixed $\lambda$.

If $N=1$ we also write $e(x,y;\lambda)$ for $\mathbf{e}(x,y;\lambda)$. Safarov (\cite[Cor. 3.1]{Safarov})
proved for every $x\in U$ and all $\lambda>0$ that
\begin{equation}
e(x,x;\lambda)\leq C_d^{(1)}\lambda^{d/2}+\frac{C_d^{(2)}}{\rho(x)}\Big(\lambda^{1/2}+\frac{C_d^{(3)}}{\rho(x)}\Big)^{d-1},\label{Safarov}\end{equation}
where $C_d^{(1)}$, $C_d^{(2)}$, $C_d^{(3)}$ are universal constants given respectively by
$C_d^{(1)}=\omega_d(2\pi)^{-d}$ with $\omega_d$ denoting the volume of the unit ball in $\mathbb{R}^d$, $C_d^{(2)}=
dC_d^{(1)}(2\pi^{-1}(C_d^{(3)})^2+C_d^{(3)})$, $C_d^{(3)}\leq 2m_d3^{\frac{1}{2m_d}}$, where $m_d=\lceil\frac{d+1}{2}\rceil$ (see \cite[Lemma 2.6]{Safarov}). We should mention that  Safarov originally established (\ref{Safarov})
 by understanding $e(x,x;\lambda)$  as the integral kernel of $\frac{\Pi(\lambda-0)+\Pi(\lambda+0)}{2}$.
But the right hand side of (\ref{Safarov}) is a continuous function of $\lambda>0$,  (\ref{Safarov})
also holds for our choice of the spectral function.

The key points for proving (\ref{Safarov}) are the fact (see \cite[Lemma 2.7, Cor. 3.1]{Safarov}) that
$\chi_{+}(\lambda) e(x,x;\lambda^2)$ is a non-decreasing function of $\lambda$ on $\mathbb{R}$, and
the cosine Fourier transform of \[(C_d^{(1)})^{-1}\cdot\frac{d}{d\lambda}\big(\chi_{+}(\lambda) e(x,x;\lambda^2)\big)\] coincides 
on the interval $(-\rho(x),\rho(x))$ with the cosine Fourier transform of
 $d\lambda_{+}^{d-1}$.
 Here $\chi_{+}$ is the characteristic function of the positive axis. The latter property can be seen from the finite propagation speed for the wave equation.
 
In the vector-valued situation, we claim as (non-negative) self-adjoint matrices, 
 \begin{equation}
 \mathbf{e}(x,x;\lambda)\leq \Big(C_d^{(1)}\lambda^{d/2}+\frac{C_d^{(2)}}{\rho(x)}\Big(\lambda^{1/2}+\frac{C_d^{(3)}}{\rho(x)}\Big)^{d-1}\Big) \id \label{Safarov3}\end{equation} 
for every $x\in U$ and all $\lambda>0$. To this end
we see once again from the finite propagation speed for the wave equation that for each fixed unit vector $v\in\mathbb{C}^N$, 
the cosine Fourier transform of
 \[(C_d^{(1)})^{-1}\cdot\frac{d}{d\lambda}\big(\chi_{+}(\lambda)\langle \delta_x^{(v)}, \Pi(\lambda^2)\delta_x^{(v)}\rangle\big)\] 
 coincides 
on the interval $(-\rho(x),\rho(x))$ with the cosine Fourier transform of
 $d\lambda_{+}^{d-1}$. 
 Also, $\chi_{+}(\lambda)\langle \delta_x^{(v)}, \Pi(\lambda^2)\delta_x^{(v)}\rangle$ is a non-decreasing function of $\lambda$ on $\mathbb{R}$.
 So similar to (\ref{Safarov}) we have
\[
\langle\delta_x^{(v)}, \Pi(\lambda)\delta_x^{(v)}\rangle\leq  C_d^{(1)}\lambda^{d/2}+\frac{C_d^{(2)}}{\rho(x)}\Big(\lambda^{1/2}+\frac{C_d^{(3)}}{\rho(x)}\Big)^{d-1} \ \ \ (x\in U, \lambda>0),\]
which proves (\ref{Safarov3}).
For simplicity, applying H\"{o}lder's and Young's inequalities to the right hand side of (\ref{Safarov3}) gives for every $x\in U$ and all $\lambda>0$ that
\begin{equation}
\mathbf{e}(x,x;\lambda)\leq (C_d^{(4)}\lambda^{d/2}+C_d^{(5)}\rho(x)^{-d})  \id,\label{Safarov2}
\end{equation}
where $C_d^{(4)}=(C_d^{(1)}+\frac{d-1}{d}2^{d-2}C_d^{(2)})$,  $C_d^{(5)}=2^{d-2}C_d^{(2)}((C_d^{(3)})^{d-1}+\frac{1}{d})$.

According to the functional calculus of self-adjoint operators, we have
\[(1+\Delta_U)^{-m}=\int_{0}^{\infty}\frac{1}{(1+\lambda)^m} d\Pi(\lambda).\]
For $m>0$ this integral can be understood as an operator-valued integral that
converges in the strong operator topology. For the purposes of this paper it is enough to understand it
in the weak sense as a statement about quadratic forms. 
For $m>\frac{d}{2}$ 
the integral kernel $\mathbf{G}_U^{(m)}(x,y)$ of
$(1+\Delta_U)^{-m}$ is continuous on $U \times U$ and we have
$$
 \mathbf{G}_U^{(m)}(x,x) = \int_{0}^{\infty}\frac{1}{(1+\lambda)^m} d\mathbf{e}(x,x;\lambda).
$$
Pointwise convergence of the integral can easily seen as follows. Choose $s \in \mathbb R$ such that $m>s>\frac{d}{2}$. The operator
$(1+\Delta_U)^{s/2}$ commutes with the spectral measure and $(1+\Delta_U)^{s/2} (1+\Delta_U)^{-m} (1+\Delta_U)^{s/2}$ is bounded. Therefore,
 its integral spectral representation converges in the sense of quadratic forms. Since $(1+\Delta_U)^{-s/2}$ maps
$L^2(U; \mathbb C^N)$ to $H^s_{loc}(U; \mathbb C^N)$ it extends by duality to a continuous map $H^{-s}_{comp}(U; \mathbb C^N)$ to $L^2(U; \mathbb C^N)$.
We conclude that
\[(1+\Delta_U)^{-m}=\int_{0}^{\infty}\frac{1}{(1+\lambda)^m} d\Pi(\lambda)\]
converges in the sense of quadratic forms on $H^{-s}_{comp}(U; \mathbb C^N)$. By the Sobolev embedding theorem $\delta^{(v)}_x$ is in $H^{-s}_{comp}(U; \mathbb C^N)$.
Considering $m>\frac{d}{2}$, (\ref{Safarov2}), $\mathbf{e}(x,x;0)\geq0$, and the following equivalent representation of the classical Beta function
\[B(\alpha,\beta)=\int_0^{\infty}\frac{\lambda^{\alpha-1}}{(1+\lambda)^{\alpha+\beta}}d\lambda\ \ \ (\mathrm{Re}(\alpha)>0,\ \mathrm{Re}(\beta)>0),\]
one can use integration by parts to get
\begin{align}
\mathbf{G}_U^{(m)}(x,x)&=\int_{0}^{\infty}\frac{1}{(1+\lambda)^m} d\mathbf{e}(x,x;\lambda)\nonumber\\
&= m\int_{0}^{\infty}\frac{\mathbf{e}(x,x;\lambda)}{(1+\lambda)^{m+1}} d\lambda-\mathbf{e}(x,x;0)\nonumber\\
&\leq \Big(m\int_{0}^{\infty}\frac{C_d^{(4)}\lambda^{d/2}+C_d^{(5)}\rho(x)^{-d}}{(1+\lambda)^{m+1}} d\lambda\Big) \id\nonumber\\
&=\big(mC_d^{(4)}B(1+\frac{d}{2},m-\frac{d}{2})+C_d^{(5)}\rho(x)^{-d}\big) \id.\label{28}
\end{align}
On the other hand we have  $\mathbf{e}_0(x,x;\lambda)=C_d^{(1)}\lambda^{d/2}\id$ (see \cite[Example 3.1]{Strichartz}), 
where $\mathbf{e}_0(x,y;\lambda)$ denotes the spectral function of $\Delta_0$.
Therefore, one obtains
\begin{align}
\mathbf{G}_0^{(m)}(x,x)=\frac{\Gamma(m-\frac{d}{2})}{(4\pi)^{\frac{d}{2}}(m-1)!}\id.\label{29}
\end{align}
Finally, by considering (\ref{28}) and (\ref{29}) and by introducing
\[C_d^{(6)}=mC_d^{(4)}B(1+\frac{d}{2},m-\frac{d}{2})+\frac{\Gamma(m-\frac{d}{2})}{(4\pi)^{\frac{d}{2}}(m-1)!},\]
we can deduce from Theorem \ref{Theorem21} that

\begin{theorem}\label{Theorem22} 
The following pointwise estimate holds for the heat kernel:
\begin{align*}
\|\mathbf{K}_U(x,y;t)-\mathbf{K}_0(x,y;t)\|\leq \big(C_d^{(5)}\rho(x,y)^{-d}+C_d^{(6)}\big)\cdot \frac{J_{m}(\rho(x)+\rho(y);t)}{2\sqrt{\pi t}}.
\end{align*}
\end{theorem}

\subsection{Optimizing cut-off functions}\label{section23}

This section is devoted to proving Theorem \ref{Theorem11}. To this end it suffices 
to bound  $J_m(R;t)$ for $R=\rho(x)+\rho(y)$. In general we suppose $R>0$.
The Hermite polynomials
\[H_n(s)=(-1)^ne^{s^2}\frac{d^n}{ds^n}e^{-s^2}\ \ \ (n=0,1,2,\ldots)\]
can be written as
\[H_n(s)=\sum_{k=0}^{\lfloor \frac{n}{2}\rfloor}\frac{(-1)^kn!}{k!(n-2k)!}(2s)^{n-2k},\]
from which it is easy to deduce that
\[\frac{d^n}{ds^n}(e^{-\frac{s^2}{4t}})=\sum_{k=0}^{\lfloor \frac{n}{2}\rfloor}\frac{(-\frac{1}{2})^n(-1)^kn!}{k!(n-2k)!}t^{k-n}s^{n-2k}e^{-\frac{s^2}{4t}} \ \ \ (n=0,1,2,\ldots).
\]
Consequently, by Leibniz's rule one gets for any non-negative integer $n\leq2m$ that
\[
\frac{d^n}{ds^n}\Big(\psi(s)e^{-\frac{s^2}{4t}}\Big)=\sum_{j=0}^n\sum_{k=0}^{\lfloor \frac{j}{2}\rfloor}{n \choose j}\psi^{(n-j)}\frac{(-\frac{1}{2})^j(-1)^kj!}{k!(j-2k)!}t^{k-j}s^{j-2k}e^{-\frac{s^2}{4t}}.
\]

To optimize the choice of cutoff functions we first let $\psi_0$ denote a fixed real-valued function 
in $C^{2m}(\mathbb{R})$ such that $\psi_0(s)=0$ for $s\leq0$ and $\psi_0(s)=1$ for $s\geq1$.
Later on we will give concrete examples of $\psi_0$ and thus
\[M_j(\psi_0)=\max_{0\leq s\leq 1}\Big|\frac{d^j\psi_0}{ds^j}(s)\Big|\ \ \ (j=0,1,\ldots,2m)
\]
can be explicitly determined. Then for any  $0<\epsilon_1<\epsilon_2<R$ 
 define \[\psi_{\epsilon_1,\epsilon_2}(s)=\psi_0\Big(\frac{|s|-\epsilon_1}{\epsilon_2-\epsilon_1}\Big),\] 
 which is an even function in $C^{2m}(\mathbb{R})$
with $\mathrm{Supp}(1-\psi_{\epsilon_1,\epsilon_2})\subset(-R,R)$. 
We let the parameters $\epsilon_1,\epsilon_2$ (depending on both $R$ and $t$) behave in the following way:
\begin{itemize}
\item $\epsilon_2\rightarrow R$,
\item $\epsilon_2-\epsilon_1\equiv\frac{2t}{R}$.
\end{itemize}
With the help of Lemma \ref{lemma24}, it is not hard to show that if $0<t\leq\frac{R^2}{8}$, then
\begin{align}\label{basic}
\lim_{\epsilon_2\rightarrow R}\int_{\mathbb{R}}
\Big|\frac{d^n}{ds^n}(\psi_{\epsilon_1,\epsilon_2}(s)e^{-\frac{s^2}{4t}})\Big|ds
\leq Z(n,\psi_0,R;t)e^{-\frac{R^2}{4t}}\ \ \ (n\leq2m),
\end{align}
where $Z(n,\psi_0,R;t)$  is short for the rational function
\[\sum_{k=0}^{\lfloor \frac{n}{2}\rfloor}
\frac{n!\lceil \frac{n-2k-1}{2}\rceil! M_0(\psi_0)e^2R^{n-2k-1}}{2^{2n-2k-1}k!(n-2k)!}t^{1+k-n}+
  \sum_{j=0}^{n-1}\sum_{k=0}^{\lfloor \frac{j}{2}\rfloor}\frac{n!M_{n-j}(\psi_0)eR^{n-2k-1}}{2^{n-2}k!(j-2k)!(n-j)!}t^{1+k-n}.\]
In general, it follows straightforward from Leibniz's rule and (\ref{basic}) that
\begin{theorem}\label{Theoremop} Suppose $0<t\leq\frac{R^2}{8}$. Then
\[J_m(R;t)\leq\sum_{n=0}^m{m\choose n}Z(2n,\psi_0,R;t)e^{-\frac{R^2}{4t}}.\]
\end{theorem}

Theorem \ref{Theorem11} is an immediate consequence of Theorems \ref{Theorem22}, \ref{Theoremop} with $m=\lceil\frac{d+1}{2}\rceil$.

 \begin{lemma}\label{lemma24}
If $\beta$ is a non-negative integer  and if $\rho\geq2\sqrt{t}$, then
\[
\int_{\rho}^{\infty}s^{\beta}e^{-\frac{s^2}{4t}}ds\leq 2 e \Big\lceil \frac{\beta-1}{2}\Big\rceil! \; \rho^{\beta-1}te^{-\frac{\rho^2}{4t}}.
\]
\end{lemma}

\begin{proof}
Note first\[\int_{\rho}^{\infty}s^{\beta}e^{-\frac{s^2}{4t}}ds=2^{\beta}t^{\frac{\beta+1}{2}}\Gamma(\frac{\beta+1}{2},\frac{\rho^2}{4t}),\]
where $\Gamma(\cdot,\cdot)$ is the upper incomplete Gamma function. 
If $\frac{\beta+1}{2}$ is a positive integer then it is known that
$\Gamma(\frac{\beta+1}{2},r)=(\frac{\beta-1}{2})!\,e^{-r}\sum _{k=0}^{\frac{\beta-1}{2}}{\frac {r^{k}}{k!}}$
for all $r>0$. This partially proves the lemma simply by considering $\frac{\rho^2}{4t}\geq1$. If $\frac{\beta+1}{2}$ is a positive half-integer
then we can use $\Gamma(\frac{\beta+1}{2},r)\leq\frac{1}{\sqrt{r}}\Gamma(\frac{\beta+2}{2},r)$ and the previous explicit formula for 
$\Gamma(\frac{\beta+2}{2},r)$ to prove the remaining part of the lemma. This finishes the proof.
\end{proof}

Although there are many test functions for $\psi_0$,
we use an interpolating polynomial because $M_j(\psi_0)$ can be 
determined rather easily.
For any $n\in\mathbb{N}$, there exists a unique polynomial $P_n$
of degree $\leq 2n+1$ such that $P_n(0)=0$, 
$P_n(1)=1$, and
\[\frac{d^i}{ds^i}P_n\Big|_{s=0}=\frac{d^i}{ds^i}P_n\Big|_{s=1}=0\ \ \ (1\leq i\leq n).\]
We then define a function $\widetilde{P}_n$ on $\mathbb{R}$ such that it agrees with $P_n$ on $[0,1]$, 
equals 0 on $(-\infty,0)$, and  equals 1 on $(1,\infty)$. It is easy to check that $\widetilde{P}_n\in C^n(\mathbb{R})$.
This means that one can set $\psi_0=\widetilde{P}_{2m}$. A few examples of $P_n$ are listed below:
\begin{align*}
P_1(s)&=3s^2-2s^3,\\
P_2(s)&=10s^3-15s^4+6s^5,\\
P_3(s)&=35s^4-84s^5+70s^6-20s^7,\\
P_4(s)&=126s^5-420s^6+540s^7-315s^8+70s^9.
\end{align*}

 \section{Dirichlet boundary conditions}\label{section3}

We denote by $K_{U}^{(D)}(x,y;t)$  the Dirichlet heat kernel for an open set $U\subset\mathbb{R}^d$. Michiel van den Berg's (\ref{BestBerg}) gives
\begin{equation}\label{formula32}
|K_{U}^{(D)}(x,y;t)-K_{0}(x,y;t)|\leq (4\pi t)^{-d/2} \exp({-\frac{|x-y|^2+4\delta^2}{4t}})\sum_{j=1}^{d}\frac{2^j\delta^{2j-2}}{(j-1)!t^{j-1}}.
\end{equation}
To compare,  Theorem \ref{Theorem11} is a slight improvement of (\ref{formula32}) for the short-time diagonal elements of the Dirichlet heat kernel
 if $d\geq5$.

 It is known that
$G_U^{(m)}(x,x)\leq G_0^{(m)}(x,x)$  for any $x\in U$, where $G_U^{(m)}$ is interpreted in accordance with the choice that $\Delta_U$
denotes the Dirichlet Laplacian on $U$. Thus it follows from Theorem \ref{Theorem21} and (\ref{29}) that
\begin{equation}\label{formula31}|K_{U}^{(D)}(x,y;t)-K_{0}(x,y;t)|
\leq
\frac{\Gamma(m-\frac{d}{2})}{(4\pi)^{\frac{d}{2}}(m-1)!\sqrt{\pi}}
\cdot \frac{J_m(\rho(x)+\rho(y);t)}{\sqrt{t}}.\end{equation}

We remark that two other estimates by Michiel van den Berg (\cite{Berg1}) reading 
\begin{align}
|K_{U}^{(D)}(x,x;t)-K_{0}(x,x;t)|
&\leq
\frac{2d}{(4\pi t)^{\frac{d}{2}}}
\exp(-\frac{\rho(x)^2}{dt}),\\
|K_{U}^{(D)}(x,y;t)-K_{0}(x,y;t)|&\leq \frac{2d}{(4\pi t)^{d/2}}\exp(-(3-2\sqrt{2})\frac{(\max\{\rho(x), \rho(y)\})^2}{dt}),
\end{align}
have been widely used in the study of short-time asymptotics of the heat trace  (see e.g. \cite{Berg87,Berg88,ST}) and 
some other related problems (see e.g. \cite{Berg07}).

  \section{Neumann boundary conditions}\label{section4}
  
 Let $K_{U}^{(N)}(x,y;t)$ denote the Neumann heat kernel for a smooth bounded open set $U\subset\mathbb{R}^d$.  
  As an application of Theorem \ref{Theorem11}   (or Theorem \ref{Theorem41} with $m=\lceil\frac{d+1}{2}\rceil$),  
  there exists a positive function $g$ on $U$   such that if $0<t\leq\frac{\rho(x)^2}{2}$ then
\begin{equation}\label{formula45}
K_{U}^{(N)}(x,x;t)\leq (4\pi t)^{-d/2}+g(x)\cdot t^{-\alpha}\cdot \exp({-\frac{\rho(x)^2}{t}}),
\end{equation}
where $\alpha=2\lceil\frac{d+1}{2}\rceil-\frac{1}{2}$. This answers a question raised by Lacey (\cite{Lacey}) who conjectured  that 
for the class of smooth bounded strictly star-shaped domains (\ref{formula45})
holds for some $\alpha>\frac{d}{2}$ as long as time $t$ is sufficiently small.
 Lacey also asked 
to extend  the main result in \cite{Lacey} to unbounded domains, domains with non-smooth boundary, 
or more general boundary conditions. Because of Theorem \ref{Theorem11} this is indeed doable
for the diagonal element of the corresponding Neumann heat kernel.

 In the rest of the section we also give a replacement of  (\ref{28}) for $G_U^{(m)}(x,x)$ 
 without using Safarov's  estimate (\ref{Safarov}). 
Here $G_U^{(m)}$ is interpreted in accordance with the choice that $\Delta_U$
denotes the Neumann Laplacian on $U$. This can be done by appealing to partial domain monotonicity 
of the Neumann heat kernel. 

For simplicity we  assume that $U\subset\mathbb{R}^d$ is  a smooth bounded open set.
Note first (see e.g. \cite[(3.33)]{Borisov}, \cite[$\S$3.4]{Davies})
\begin{equation}
\label{41}
G_U^{(m)}(x,x)=\frac{1}{(m-1)!}\int_0^{\infty}t^{m-1}e^{-t}K_U^{(N)}(x,x;t)dt,
\end{equation}
  which means that we need instead to prepare suitable upper bounds for $K_U^{(N)}(x,x;t)$.
  In contrast to the Dirichlet boundary problems,
  there does not exist a general domain monotonicity principle (\cite{Bass}) claiming for any $U_2\subset U_1$ that
  \[K_{U_1}^{(N)}(x,y;t)\leq K_{U_2}^{(N)}(x,y;t)\ \ \ ((x,y,t)\in U_2\times U_2\times \mathbb{R}^{+}).\]
  Even though, Kac's original idea (\cite{Kac}) of comparing $K_U^{(D)}(x,x;t)$ with $K_{B_x}^{(D)}(x,x;t)$
 where $B_x$ is chosen here\footnote{To be precise, Kac (\cite{Kac})  set $B_x$ to be the largest open cube contained in $U$.} to be the ball in $\mathbb{R}^d$ with center $x$ and radius $\rho(x)$, still works for the Neumann boundary problems. 
 This is exactly a result by Kendall (\cite{Kendall}, see also \cite{Pascu}; if $U$ is convex then see \cite{Chavel})
 stating
\begin{equation}\label{42}
K_U^{(N)}(x,x;t)\leq K_{B_x}^{(N)}(x,x;t)\ \ \ ((x,t)\in U\times\mathbb{R}^{+}),
\end{equation}
which combined with (\ref{41}) yields
\begin{equation}\label{43}
G_U^{(m)}(x,x)\leq G_{B_x}^{(m)}(x,x)\ \ \  (x\in U).
\end{equation}
 Let
 $\mathbb{U}_d(x,y;t)$
  denote the Neumann heat kernel for the $d$-dimensional unit ball.    
  The Pascu-Gageonea resolution (\cite{PG}) of the Laugesen-Morpurgo conjecture (\cite{LM}) says that
\begin{equation}\label{444}
\mathbb{U}_d(x,x;t)<\mathbb{U}_d(y,y;t)
\end{equation}
 holds for all $t>0$ and all $x,y$ in the $d$-dimensional unit ball with $|x|<|y|$. This result implies that
 \begin{equation}
 \label{45}
 \mathbb{U}_d(0,0;t)<\frac{\mathrm{Tr}(e^{-t\Delta_d^{(N)}})}{\omega_d},
 \end{equation}
  where $\Delta_d^{(N)}$ is short for the Neumann Laplacian on the $d$-dimensional unit ball.

Now let $x\in U$ be fixed.  It is straightforward to verify that
\begin{equation}\label{formula46}
K_{B_x}^{(N)}(x,x;t)=\frac{\displaystyle\mathbb{U}_d\big(0,0;\frac{t}{\rho(x)^2}\big)}{\rho(x)^d}.
\end{equation}
Hence by considering (\ref{43}), (\ref{41}) with $U$ replaced by $B_x$,  (\ref{formula46}) and  (\ref{45}), we get
  \begin{align*}
  G_U^{(m)}(x,x)&\leq\frac{1}{(m-1)!}\cdot\int_0^{\infty}t^{m-1}e^{-t}\frac{\displaystyle\mathrm{Tr}(e^{-\frac{t}{\rho(x)^2}\Delta_d^{(N)}})}{\omega_d\rho(x)^d}dt  \nonumber\\
    &=\frac{\rho(x)^{2m-d}}{(m-1)!\omega_d}\cdot\int_0^{\infty}t^{m-1}e^{-t\rho(x)^2}\mathrm{Tr}(e^{-t\Delta_d^{(N)}})dt\nonumber\\
    &\leq\frac{\rho(x)^{2m-d}}{(m-1)!\omega_d}\cdot\Big(\int_0^{1}t^{m-1}\mathrm{Tr}(e^{-t\Delta_d^{(N)}})dt+\mathrm{Tr}(e^{-\Delta_d^{(N)}})\int_0^{\infty}t^{m-1}e^{-t\rho(x)^2}dt\Big)\nonumber\\
    &=\frac{\rho(x)^{2m-d}}{(m-1)!\omega_d}\cdot\int_0^{1}t^{m-1}\mathrm{Tr}(e^{-t\Delta_d^{(N)}})dt+ \frac{\mathrm{Tr}(e^{-\Delta_d^{(N)}})}{\omega_d\rho(x)^d},\nonumber
     \end{align*}
     where in the last inequality we have used the fact $\mathrm{Tr}(e^{-t\Delta_d^{(N)}})\leq \mathrm{Tr}(e^{-\Delta_d^{(N)}})$ for all $t\geq1$.
 This estimate together with (\ref{29})  gives from Theorem \ref{Theorem21}  the following
 \begin{theorem}\label{Theorem41} 
 Let $U\subset\mathbb{R}^d$  a smooth bounded open set and let $m\in\mathbb{N}$ with $m>\frac{d}{2}$. 
For any $t>0$ and any $x,y$ in $U$ one has
\[\big|K_{U}^{(N)}(x,y;t)-(4\pi t)^{-d/2} \exp({-\frac{|x-y|^2}{4t}})\big|\leq
N_d(x,y)\cdot \frac{J_m(\rho(x)+\rho(y);t)}{2\sqrt{\pi t}},\]
where
\begin{align*}
N_d(x,y)&=\frac{\displaystyle\int_0^{1}t^{m-1}\mathrm{Tr}(e^{-t\Delta_d^{(N)}})dt}{(m-1)!\omega_d}\cdot(\max\{\rho(x),\rho(y)\})^{2m-d}+\frac{\Gamma(m-\frac{d}{2})}{(4\pi)^{\frac{d}{2}}(m-1)!}+\\
&\ \ \ \ \frac{\mathrm{Tr}(e^{-\Delta_d^{(N)}})}{\omega_d}\cdot(\min\{\rho(x),\rho(y)\})^{-d}.
\end{align*}
\end{theorem}



\end{document}